\documentclass[11pt,a4paper]{article}
\usepackage[T1]{fontenc}
\usepackage[utf8]{inputenc}
\usepackage{lmodern}
\usepackage[margin=1in]{geometry}
\usepackage{amsmath,amssymb,amsthm}
\usepackage{microtype}
\usepackage[hidelinks]{hyperref}
\usepackage[affil-it]{authblk}
\newtheorem{theorem}{Theorem}[section]
\newtheorem{lemma}[theorem]{Lemma}

\theoremstyle{definition}
\newtheorem{definition}[theorem]{Definition}
\theoremstyle{remark}

\numberwithin{equation}{section}

\newcommand{\cH}{\mathcal{H}}

\title{On the distinct maximal-clique sizes in $k$-uniform hypergraphs}
\author[1]{Jiabao Yang\footnote{Email: \texttt{jbyang1215@nju.edu.cn}}}
\author[2]{Lixiang Yu \footnote{Email: \texttt{lixiangyu031862@163.com}}}
\author[3,4]{Leilei Zhang\footnote{Corresponding author. Email: \texttt{mathdzhang@163.com}}}
\affil[1]{\small School of Mathematics, Nanjing University, Nanjing 210093, China}
\affil[2]{\small Department of Mathematics, East China Normal University, Shanghai, 200241, China}
\affil[3]{\small School of Mathematics and Statistics, Central China Normal University, Wuhan 430079, China}
\affil[4]{footnotesize Faculty of Environment and Information Sciences, Yokohama National University, Yokohama 240-8501, Japan}
\date{}
\begin{document}
	\maketitle
	 
 \begin{abstract}
 	Let $g(n,k)$ be the maximum number of distinct sizes of maximal
 	cliques in an $n$-vertex $k$-uniform hypergraph, and let
 	$f(n,k)=n-g(n,k)$. We determine the asymptotic order of $f(n,k)$
 	for every fixed integer $k\ge 3$. Define
 	$L_2(x)=\max\{2,\log_2(\max\{1,x\})\}$, and, for $j\ge 3$,
 	let $L_j(x)$ be the least number of iterations of $L_{j-1}$
 	needed to reach a value at most $16$. We prove that
 	\[
 	f(n,k)=\Theta_k(L_k(n)).
 	\]
 	In particular, $f(n,3)=\Theta(\log^{*}n)$, where $\log^{*}n$
 	denotes the iterated logarithm. We also determine the asymptotic
 	behaviour of the layered-tree threshold $c(n,k)$ arising from
 	Gao's insertion-tree method:
 	\[
 	c(n,k)=\log_2 L_k(n)+O_k(1).
 	\]
 	Consequently, $f(n,k)=\Theta_k\!\left(2^{c(n,k)}\right)$.
 	The upper bound follows from a recursive selector construction
 	combined with an elementary sumset argument. For the lower bound,
 	we use maximality to obtain a compression estimate for leaf
 	children in the insertion tree, and then apply a multilevel
 	counting argument. Both arguments apply to all fixed $k\ge 3$.
 \end{abstract}
    \medskip
	\noindent\textbf{Keywords:} hypergraph; maximal clique; clique size; layered tree
	
	\smallskip
	\noindent\textbf{2020 Mathematics Subject Classification:} 05C65, 05C69, 05C35
	
	\section{Introduction} 	\label{sec:introduction}
	
	We study how many different sizes maximal cliques can have. All hypergraphs in this paper are finite and simple.  Throughout this paper $\log$ will denote logarithm to the base $2$. For an integer $k\ge2$, a $k$-uniform hypergraph is a pair $\cH=(V(\cH),E(\cH))$ with $E(\cH)\subseteq\binom{V(\cH)}{k}$.  A set $X\subseteq V(\cH)$ is \emph{complete} if every $k$-element subset of $X$ belongs to $E(\cH)$. 
	A \emph{clique} is a complete set that is maximal under inclusion. 
	Thus every set with fewer than $k$ vertices is complete. 
	Let $g(n,k)$ be the largest number of distinct clique sizes in an $n$-vertex $k$-uniform hypergraph, and define $f(n,k)=n-g(n,k)$.
	
	We write
	\[
	g(n,k)=\max_{|V(H)|=n}
	\bigl|\{|X|:X\text{ is a maximal clique of }H\}\bigr|,
	\qquad f(n,k)=n-g(n,k).
	\]
	Thus \(f(n,k)\) is the smallest possible number of sizes in
	\(\{1,\ldots,n\}\) that do not occur as maximal-clique sizes.
	
	All logarithms have base \(2\). In all asymptotic statements, \(n\)
	tends to infinity and \(k\) is fixed. A constant with subscript \(k\)
	may depend on \(k\). An interval used in a sumset contains only integers.
	For sets of integers \(A\) and \(B\), their sumset is
	\(A+B=\{a+b:a\in A,\ b\in B\}\).
	
	Fix \(K=16\), and define
	\[
	\begin{aligned}
		L_2(x)&=\max\{2,\log_2(\max\{1,x\})\},\\
		L_j(x)&=\min\{t\in\mathbb Z_{\ge0}:L_{j-1}^{\circ t}(x)\le K\}
		\qquad(j\ge3).
	\end{aligned}
	\]
	Here \(F^{\circ t}\) means \(t\) repeated applications of \(F\), and
	\(F^{\circ0}(x)=x\). The functions are defined for \(x\ge0\), including
	when an iterate equals zero. We call \(L_j\) the rank of level \(j\).
	In particular, \(L_3\) is an iterated logarithm.
	Section~\ref{sec:lemmas} proves the properties of these functions that
	we need.
	
	A \((d,C)\)-layered tree is a rooted tree of depth at most \(d\), with
	an order \(v_0,\ldots,v_{N-1}\) in which the root comes first, every
	parent comes before its children, and \(\deg(v_i)\le2^{C+i}\).
	For nonnegative integers \(d,C\), let \(N_0(d,C)\) be the largest
	number of vertices in such a tree. We prove below that this number
	is finite. Define
	\[
	c(n,k)=\min\{C\in\mathbb Z_{\ge0}:n-C\le N_0(k-1,C)\}.
	\]
	
	\begin{theorem}
		\label{thm:main}
		For every fixed integer \(k\ge3\),
		\[
		f(n,k)=\Theta_k(L_k(n)),\qquad
		c(n,k)=\log_2 L_k(n)+O_k(1).
		\]
		It follows that \(f(n,k)=\Theta_k(2^{c(n,k)})\).
	\end{theorem}
	
	The insertion tree and the depth-first-search argument used here are
	due to Gao~\cite{Gao}. Yang and Zhang~\cite{YangZhang} proved
	\(f(n,3)=\Theta(\log^* n)\) and proposed the relation with
	\(2^{c(n,k)}\). We include proofs of all the tree facts needed here.
	Our upper and lower bounds both cover \(k=3\).
	
	Section~\ref{sec:lemmas} gives the main lemmas and their proofs.
	We first study the rank functions and layered trees, then give the
	construction for the upper bound and the tree lemmas for the lower bound.
	Section~\ref{sec:proof} combines these results to prove
	Theorem~\ref{thm:main}.
	
	\section{Main lemmas}
	\label{sec:lemmas}
	
	We begin with the rank functions defined in the introduction.
	The next lemma shows that they are well defined and change by only a
	bounded amount under the operations used later.
	
	\begin{lemma}[Properties of the rank functions]
		\label{lem:ranks}
		For every \(j\ge3\), the function \(L_j\) is finite, nondecreasing and
		unbounded, and \(0\le L_j(x)\le L_2(x)\). In particular,
		\(L_j(x)=o(x)\) as \(x\to\infty\). For \(x>K\),
		\begin{equation}
			L_j(x)=1+L_j(L_{j-1}(x)).
			\label{eq:rank-recurrence}
		\end{equation}
		For fixed positive constants \(a,A\),
		\[
		L_j(ax+A)=L_j(x)+O(1),\qquad
		L_j(2^{Ax^a})=L_j(x)+O(1).
		\]
		A fixed positive power, or a fixed number of exponential operations
		applied to such a power, also changes \(L_j\) by only a bounded amount.
	\end{lemma}
	
	\begin{proof}
		For \(y>K\), we have \(L_2(y)=\log_2 y\le y-1\). Thus the iteration
		defining \(L_3\) stops, and \(L_3(y)\le\lceil y-K\rceil\) for \(y>K\).
		If \(x>K\) and \(z=L_2(x)\), then either \(z\le K\), in which case
		\(L_3(x)=1\le z\), or
		\[
		L_3(x)=1+L_3(z)\le1+\lceil z-K\rceil\le z.
		\]
		For \(x\le K\), the inequality \(L_3(x)\le L_2(x)\) is immediate.
		Inductively, if \(L_{j-1}\le L_2\), its iterates reach \([0,K]\) no later
		than the corresponding iterates of \(L_2\). Hence \(L_j\le L_3\le L_2\).
		This proves finiteness and the upper bound for all ranks.
		Monotonicity follows from monotonicity of the function being iterated.
		If that function is unbounded, each of its fixed finite compositions is
		unbounded. Its stopping time is therefore unbounded as well. Starting
		with \(L_2\), this proves unboundedness inductively. The recurrence follows
		directly from the definition.
		
		Write \(F=L_{j-1}\). For sufficiently large \(x\),
		\[
		F(ax+A)\le L_2(ax+A)\le x,\qquad
		F(x)\le L_2(x)\le ax+A.
		\]
		Applying the recurrence and monotonicity in both directions gives
		\(\lvert L_j(ax+A)-L_j(x)\rvert\le1\).
		For \(y=2^{Ax^a}\), eventually \(y\ge x\) and
		\[
		F^{\circ2}(y)\le L_2^{\circ2}(y)=\log_2(Ax^a)\le x.
		\]
		Thus \(0\le L_j(y)-L_j(x)\le2\). Similarly, \(F(x^a)\le x\) and
		\(F(x)\le x^a\) for large \(x\), which proves stability under positive
		powers. Applying the exponential estimate a fixed number of times proves
		the last assertion.
	\end{proof}
	
	\begin{lemma}[Combining rank bounds]
		\label{lem:combine-ranks}
		Fix \(m\ge3\) and constants \(A,K_0>0\). Suppose
		\[
		T=U_0\le U_1\le\cdots\le U_N,\qquad N\le2^{AT},
		\]
		and
		\[
		L_m(U_{i+1})\le L_m(U_i)+K_0\qquad(0\le i<N).
		\]
		For all sufficiently large \(T\),
		\[
		L_{m+1}(U_N)\le L_{m+1}(T)+O_{m,A,K_0}(1).
		\]
	\end{lemma}
	
	\begin{proof}
		Summing the hypotheses and using \(L_m(T)\le L_2(T)\) gives
		\[
		L_m(U_N)\le L_m(T)+K_0N\le2^{A'T}
		\]
		for a suitable fixed \(A'>0\). By the recurrence, monotonicity and stability,
		\[
		\begin{aligned}
			L_{m+1}(U_N)
			&=1+L_{m+1}(L_m(U_N))\\
			&\le1+L_{m+1}(2^{A'T})
			=L_{m+1}(T)+O(1).
		\end{aligned}
		\]
	\end{proof}
	
	We next estimate the size of layered trees. We first restate the
	definition so that the ordering condition is explicit.
	
	\begin{definition}
		\label{def:layered-tree}
		For nonnegative integers \(d,C\), a \((d,C)\)-layered tree is a rooted tree
		with an ordering \(v_0,v_1,\ldots,v_{N-1}\) such that \(v_0\) is the root,
		every vertex has depth at most \(d\), the parent of \(v_i\) precedes \(v_i\)
		for \(i>0\), and
		\[
		\deg_T(v_i)\le2^{C+i}\qquad(0\le i<N).
		\]
	\end{definition}
	
	The parent-before-child condition agrees with Gao's requirement that each
	vertex other than the root have a preceding neighbor: in a tree, the earlier
	vertices then form a connected subtree containing the root. A depth-first
	search order, abbreviated to DFS order, visits each rooted branch in one
	consecutive block.
	
	For integers \(D,h\ge0\), define \(B_0(D)=1\). For \(h\ge1\), put
	\begin{equation}
		\begin{gathered}
			x_0=1,\qquad x_{s+1}=x_s+B_{h-1}(D+x_s)
			\quad(0\le s<2^D-1),\\
			B_h(D)=x_{2^D-1}.
		\end{gathered}
		\label{eq:branch-recursion}
	\end{equation}
	Here \(B_h(D)\) is the largest order of a DFS branch of height at most
	\(h\) whose root has an additional edge to a parent outside the branch,
	and whose vertex of local index \(i\) has total degree at most \(2^{D+i}\).
	Indeed, its root has at most \(2^D-1\) children. A child branch starting
	after \(x_s\) vertices has parameter \(D+x_s\). Induction shows that each
	\(B_h\) is nondecreasing in \(D\). Filling a child branch to its maximum
	increases the available indices for all later branches, so the displayed
	recursion gives both a valid construction and an upper bound. In particular,
	\begin{equation}
		B_1(D)=2^D.
		\label{eq:height-one}
	\end{equation}
	
	\begin{lemma}[Finite size and DFS order]
		\label{lem:finite-dfs}
		Every \((d,C)\)-layered tree has order at most \(B_d(C+1)\). Thus a
		largest such tree exists. Every ordering of a largest such tree that
		satisfies Definition~\ref{def:layered-tree} is a DFS order.
	\end{lemma}
	
	\begin{proof}
		First suppose that the whole order is DFS. Regard the root as having one
		additional edge to an external parent and increase the parameter from
		\(C\) to \(C+1\). Its degree is then at most \(2^C+1\le2^{C+1}\), and all
		other degree allowances also increase. The branch interpretation of \eqref{eq:branch-recursion}
		bounds its order by \(M=B_d(C+1)\).
		
		For an arbitrary admissible ordered tree, consider its longest initial
		segment compatible with a DFS of that tree. Suppose the next vertex is
		\(v_j\), so the compatible segment has \(j\) vertices. Let \(v_i\) be the
		deepest vertex on the current DFS stack that still has an unvisited child.
		Since \(v_j\) is the first failure of DFS, its parent is a strict ancestor
		of \(v_i\). There is an unvisited child \(v_{i'}\) of \(v_i\) with \(i'>j\),
		and
		\[
		\operatorname{depth}(v_i)\ge\operatorname{depth}(v_j).
		\]
		Insert a new vertex \(w\) immediately before \(v_j\) in the order, add
		\(v_iw\), delete \(v_iv_{i'}\), and add \(v_jv_{i'}\). This is the exchange
		used in Gao's DFS lemma~\cite[Lemma~3.1]{Gao}.
		
		The resulting graph is a tree: \(v_j\) lies outside the subtree at
		\(v_{i'}\), since parents precede children and \(j<i'\). The reattached
		subtree does not become deeper, and the new leaf \(w\) has allowed
		depth because \(v_i\) previously had the child \(v_{i'}\). The degree of
		\(v_i\) is unchanged. The degree of \(v_j\) increases by one, but its
		insertion index increases by one, so its allowance doubles. Every other
		old vertex either keeps its index or receives a larger one. The new leaf
		also satisfies its degree bound. Parent-before-child order is preserved.
		Thus the new tree is admissible and has one more vertex. Its old
		compatible prefix followed by \(w\) is compatible with DFS, so the longest
		compatible prefix has strictly increased.
		
		Every compatible prefix, with the induced tree, is itself a tree in DFS
		order. The first paragraph therefore bounds its length by \(M\).
		Repeating the exchange must stop: otherwise these prefix lengths
		would eventually exceed \(M\). At termination the whole order is DFS, so
		the resulting tree has at most \(M\) vertices. The original tree has no
		more vertices than this tree, proving the bound. The set of possible
		orders is therefore a nonempty bounded set of integers and has a maximum.
		A non-DFS ordering of a largest tree would admit one further exchange and
		contradict maximality.
	\end{proof}
	
	Let \(N_0(d,C)\) denote the maximum order in Definition~\ref{def:layered-tree}. The recursive
	branch with parameter \(C\) is also a valid whole tree: its root has at
	most \(2^C-1\) children and its other degree bounds agree with those in the
	definition. Together with Lemma~\ref{lem:finite-dfs}, this gives
	\begin{equation}
		B_d(C)\le N_0(d,C)\le B_d(C+1).
		\label{eq:tree-sandwich}
	\end{equation}
	
	\begin{lemma}[Rank of a branch size]
		\label{lem:branch-rank}
		For every fixed \(h\ge2\), as the integer \(D\) tends to infinity,
		\[
		L_{h+1}(B_h(D))=\Theta_h(2^D).
		\]
	\end{lemma}
	
	\begin{proof}
		For \(h=2\), equations \eqref{eq:branch-recursion} and \eqref{eq:height-one} give
		\[
		x_{s+1}=x_s+2^{D+x_s},\qquad x_1=1+2^{D+1}.
		\]
		For \(s\ge1\) and large \(D\), we have
		\[
		2^{x_s}\le x_{s+1}\le2^{2x_s}.
		\]
		The recurrence for \(L_3\) gives \(L_3(2^{x_s})=1+L_3(x_s)\), and rank
		stability gives the corresponding upper bound. Thus every update after
		the first increases \(L_3\) by at least one and at most a fixed constant.
		The first update satisfies
		\(L_3(x_1)=L_3(D)+O(1)=O(2^D)\).
		There are \(2^D-1\) updates, proving the base case.
		
		Now let \(h\ge3\) and assume the claim at \(h-1\). The recursion and
		\(B_{h-1}(y)\ge2^y\) imply
		\[
		B_{h-1}(D+x_s)\le x_{s+1}\le2B_{h-1}(D+x_s).
		\]
		By the induction hypothesis and rank stability,
		\begin{equation}
			L_h(x_{s+1})=\Theta_h(2^{D+x_s}).
			\label{eq:branch-update}
		\end{equation}
		The constants in this estimate are uniform over all update indices \(s\),
		since every argument \(D+x_s\) is at least \(D+1\).
		For \(s\ge1\), we have \(x_s\ge x_1\ge2^{D+1}\). Thus, for a fixed
		constant \(A_h>0\) and all large \(D\),
		\[
		x_s\le L_h(x_{s+1})\le2^{A_hx_s}.
		\]
		Using the rank recurrence and stability, we obtain
		\[
		1+L_{h+1}(x_s)
		\le L_{h+1}(x_{s+1})
		\le L_{h+1}(x_s)+O_h(1).
		\]
		For the first update, \eqref{eq:branch-update} gives \(L_h(x_1)=\Theta_h(2^D)\), so
		\[
		\begin{aligned}
			L_{h+1}(x_1)
			&=1+L_{h+1}(L_h(x_1))\\
			&=L_{h+1}(D)+O_h(1)=O_h(2^D).
		\end{aligned}
		\]
		Summing over the \(2^D-1\) updates proves the claim.
	\end{proof}
	
	\begin{lemma}[Size of a layered tree]
		\label{lem:tree-size}
		For every fixed \(d\ge2\),
		\[
		L_{d+1}(N_0(d,C))=\Theta_d(2^C)\qquad(C\to\infty).
		\]
	\end{lemma}
	
	\begin{proof}
		Apply monotonicity and Lemma~\ref{lem:branch-rank} to both sides of \eqref{eq:tree-sandwich}.
	\end{proof}
	
	We now give the constructions used for the upper bound.
	The graph case will start the induction on $k$.
	
	\begin{lemma}
		\label{lem:graphs}
		For every integer $n\ge 8$,
		\[
		f(n,2)\le 2\lceil\log_2 n\rceil-1.
		\]
	\end{lemma}
	
	\begin{proof}
		Put $r=\lceil\log_2 n\rceil$ and $S=n-2r$. Then $0\le S<2^r$. Define
		\[
		s_i=\left\lfloor\frac{S}{2^{r-i}}\right\rfloor\quad(0\le i\le r),
		\qquad a_i=s_i-s_{i-1}\quad(1\le i\le r).
		\]
		We have $s_0=0$, $s_r=S$, and $s_i\in\{2s_{i-1},2s_{i-1}+1\}$. Thus
		\[
		0\le a_i\le s_{i-1}+1.
		\]
		The subset sums of $a_1,\ldots,a_i$ contain all integers from $0$ to $s_i$. Indeed, this is clear for $i=0$,
		and the induction step follows because the two integer intervals $[0,s_{i-1}]$ and $a_i+[0,s_{i-1}]$
		overlap or are consecutive, and their union is $[0,s_i]$.
		
		For $1\le i\le r$, let $F_i$ be the disjoint union of a complete graph on $a_i+1$ vertices and a
		single vertex. Take the join of the graphs $F_1,\ldots,F_r$: retain their edges and add every edge
		between different graphs. The resulting graph has
		\[
		\sum_{i=1}^r(a_i+2)=S+2r=n
		\]
		vertices. Its maximal cliques are exactly the unions obtained by choosing one maximal clique
		in each $F_i$. The possible sizes in $F_i$ are $1$ and $a_i+1$; when $a_i=0$, these sizes coincide.
		Thus every size
		\[
		r,r+1,\ldots,r+S
		\]
		occurs. There are $S+1=n-2r+1$ such sizes, proving the claim.
	\end{proof}
	
	Fix $k\ge3$. In Lemmas~\ref{lem:dense-contributions} and~\ref{lem:tails}, we assume
	that there is a constant $A>0$, depending only on $k$, such that
	for every sufficiently large integer $s$,
	\begin{equation}
		f(s,k-1)\le A L_{k-1}(s)+A.\label{eq:induction-assumption}
	\end{equation}
	
	For a noncomplete $(k-1)$-uniform hypergraph $\mathcal{G}$ on $m\ge k-1$ vertices, define its
	\emph{contribution set} by
	\begin{equation}
		A(\mathcal{G})=\{0\}\cup\{m-|Y|-1:Y\text{ is a maximal clique of }\mathcal{G}\}.\label{eq:contribution-set}
	\end{equation}
	Every maximal clique of $\mathcal{G}$ has size between $k-2$ and $m-1$. Therefore
	\[
	A(\mathcal{G})\subseteq [0,m-k+1]\cap\mathbb{Z}.
	\]
	
	\begin{lemma}
		\label{lem:dense-contributions}
		Assume \eqref{eq:induction-assumption}. For every sufficiently large $m$, there is a noncomplete $(k-1)$-uniform hypergraph
		$\mathcal{G}$ on $m$ vertices such that
		\[
		0,m-k+1\in A(\mathcal{G})
		\]
		and $A(\mathcal{G})$ misses at most $A L_{k-1}(m)+A$ integers of $[0,m-k+1]$.
	\end{lemma}
	
	\begin{proof}
		Put $m_0=m-k+2$. Choose an $m_0$-vertex $(k-1)$-uniform hypergraph $\mathcal{F}$ with
		$g(m_0,k-1)$ distinct maximal-clique sizes. Add $k-2$ new vertices $a_1,\ldots,a_{k-2}$, and add no
		edges containing any of these vertices. Call the resulting hypergraph $\mathcal{G}$.
		
		Every maximal clique of $\mathcal{F}$ remains maximal in $\mathcal{G}$. In fact, it has at least $k-2$ vertices,
		so adding a new vertex would create a missing $(k-1)$-set. The set $\{a_1,\ldots,a_{k-2}\}$ is itself a
		maximal clique of $\mathcal{G}$, and its contribution is $m-k+1$. Also, $\mathcal{G}$ is noncomplete.
		
		The interval $[0,m-k+1]\cap\mathbb{Z}$ has $m_0$ elements. The maximal cliques inherited from $\mathcal{F}$
		give $g(m_0,k-1)$ distinct contributions. Hence at most $f(m_0,k-1)$ elements of this interval
		are absent from $A(\mathcal{G})$. Apply \eqref{eq:induction-assumption} and the monotonicity of $L_{k-1}$.
	\end{proof}
	
	\begin{lemma}[Filling gaps in a sumset]
		\label{lem:sumset}
		Let $A_0\subseteq[0,c]\cap\mathbb{Z}$ contain $0$ and $c$, and suppose that at most $h$
		integers in this interval are absent from $A_0$. If $B\supseteq[0,L]\cap\mathbb{Z}$, where $L$ is an integer with
		$L\ge h$, then
		\[
		A_0+B\supseteq[0,c+L]\cap\mathbb{Z}.
		\]
	\end{lemma}
	
	\begin{proof}
		Take an integer $z$ with $0\le z\le c+L$. If $z\le L$, use $z=0+z$. If $z\ge c$, use
		$z=c+(z-c)$. In the remaining case, $L<z<c$. The interval $[z-L,z]$ consists of $L+1>h$
		integers in $[0,c]$, so it contains an element $a\in A_0$. Then $z-a\in[0,L]\cap\mathbb{Z}\subseteq B$.
	\end{proof}
	
	\begin{lemma}[The selector construction]
		\label{lem:selectors}
		Let $\mathcal{G}_1,\ldots,\mathcal{G}_r$ be pairwise disjoint noncomplete $(k-1)$-uniform
		hypergraphs, with $|V(\mathcal{G}_i)|=m_i\ge k-1$, and put $n=\sum_i(m_i+1)$. There is an $n$-vertex
		$k$-uniform hypergraph whose maximal-clique sizes include
		\[
		\{n-r-d:d\in A(\mathcal{G}_1)+\cdots+A(\mathcal{G}_r)\}.
		\]
		In particular, if
		\begin{equation}
			[0,n-kr]\cap\mathbb{Z}\subseteq A(\mathcal{G}_1)+\cdots+A(\mathcal{G}_r),
			\qquad n-kr\ge0,\label{eq:selector-interval}
		\end{equation}
		then
		\begin{equation}
			f(n,k)\le kr-1.\label{eq:selector-bound}
		\end{equation}
	\end{lemma}
	
	\begin{proof}
		Add new vertices $p_1,\ldots,p_r$, one for each component. A $k$-set is an edge if it contains
		no selector or at least two selectors. A $k$-set containing exactly the selector $p_i$ is also an edge,
		unless its other $k-1$ vertices all lie in $V(\mathcal{G}_i)$; in that case it is an edge precisely when those
		$k-1$ vertices form an edge of $\mathcal{G}_i$.
		
		For each $i$, choose either the inactive state $V(\mathcal{G}_i)$ or an active state $Y_i\cup\{p_i\}$, where $Y_i$ is a
		maximal clique of $\mathcal{G}_i$. The union $X$ of the chosen states is a clique by the edge definition. It is
		maximal as well. If an ordinary vertex $x$ is omitted from an active component, maximality of
		$Y_i$ supplies a missing $(k-1)$-set in $Y_i\cup\{x\}$; adding $p_i$ gives a missing $k$-set in $X\cup\{x\}$. If
		the omitted vertex is the selector $p_i$ of an inactive component, noncompleteness of $\mathcal{G}_i$ supplies
		a missing $(k-1)$-set in $V(\mathcal{G}_i)$, and adding $p_i$ again prevents an extension of $X$.
		
		Starting with all components inactive gives size $\sum_i m_i=n-r$. An active state in
		component $i$ reduces this size by $m_i-|Y_i|-1$. This proves the claim about the realized
		sizes. Under \eqref{eq:selector-interval}, there are at least $n-kr+1$ such sizes, which proves \eqref{eq:selector-bound}.
	\end{proof}
	
	We now construct families to which Lemma~\ref{lem:selectors} applies. Call a family $(\mathcal{G}_1,\ldots,\mathcal{G}_r)$ an
	\emph{admissible $N$-tail} if its components are noncomplete $(k-1)$-uniform hypergraphs of order at
	least $k-1$,
	\[
	\sum_{i=1}^r\bigl(|V(\mathcal{G}_i)|+1\bigr)=N,\qquad N-kr\ge0,
	\]
	and its contribution sum contains $[0,N-kr]\cap\mathbb{Z}$. The extra vertex in each summand is
	reserved for its selector.
	
	\begin{lemma}
		\label{lem:tails}
		Assume \eqref{eq:induction-assumption}. There are constants $N_0$ and $R$, depending only on $k$, such that every integer
		$N\ge N_0$ has an admissible $N$-tail with at most
		\[
		2L_k(N)+R
		\]
		components.
	\end{lemma}
	
	\begin{proof}
		We first construct less efficient tails for every sufficiently large order. Let $\mathcal{E}_j$ be the
		empty $(k-1)$-uniform hypergraph on $j$ vertices. The component $\mathcal{E}_{k-1}$, together with its
		selector, uses $k$ vertices and contributes $\{0\}$, whereas $\mathcal{E}_k$ uses $k+1$ vertices and contributes
		$\{0,1\}$. For any integer $N$, choose $a\in\{0,\ldots,k\}$ such that $a\equiv-N\pmod{k+1}$, and put
		\[
		b=\frac{N-ka}{k+1}.
		\]
		If $N\ge k^2$, then $b$ is a nonnegative integer. The family of $a$ copies of $\mathcal{E}_{k-1}$ and $b$ copies of
		$\mathcal{E}_k$ is admissible: it has $r=a+b$ components and its contribution sum is exactly $[0,b]\cap\mathbb{Z}$,
		where $b=N-kr$.
		
		Fix
		\[
		\eta=\frac{1}{2(k+1)},\qquad N_0\ge2k^2,
		\]
		with $N_0$ an integer.
		The elementary tail just constructed satisfies
		\begin{equation}
			N-kr=b\ge\frac{N-k^2}{k+1}\ge\eta N\qquad(N\ge N_0).\label{eq:base-density}
		\end{equation}
		Choose constants $Q\ge N_0$ and $C\ge A/\eta$, with $Q$ an integer, and define
		\begin{equation}
			q(x)=\max\{Q,\lceil C L_{k-1}(x)+C\rceil\}.\label{eq:q-definition}
		\end{equation}
		Let $L=L_{k-1}$. By Lemma~\ref{lem:ranks}, or directly from the definition when $L=L_2$,
		the function $L$ is nondecreasing, tends to infinity, and satisfies $L(x)=o(x)$. Thus, for a fixed constant $B$, $q(x)\le B L(x)$ for all
		sufficiently large $x$. Also, $B L(By)\le y$ for all sufficiently large $y$. It follows that
		\begin{equation}
			q(q(x))\le B L(q(x))\le B L(B L(x))\le L(x)\label{eq:q-two-steps}
		\end{equation}
		for all sufficiently large $x$.
		
		Choose a fixed integer $N_1\ge\max\{N_0,Q,K\}$ sufficiently large that, for every integer
		$n>N_1$, \eqref{eq:q-two-steps} holds and
		\begin{equation}
			q(n)\le\frac n2,\qquad n\ge\frac{2k}{1-\eta},\label{eq:cutoff}
		\end{equation}
		and $m=n-q(n)-1$ is large enough for Lemma~\ref{lem:dense-contributions}. All constants and the base interval
		$[N_0,N_1]$ are now fixed.
		
		For $N_0\le n\le N_1$, use the elementary admissible tail above. For $n>N_1$, we construct a
		tail recursively, using one main component of order $m=n-q(n)-1$ supplied by Lemma~\ref{lem:dense-contributions},
		together with the already constructed $q(n)$-tail. This recursion is well defined because
		$N_0\le q(n)<n$.
		
		Write $\rho(n)$ for the number of components in the selected tail. We prove admissibility and
		the density inequality
		\begin{equation}
			n-k\rho(n)\ge\eta n\label{eq:tail-density}
		\end{equation}
		together by strong induction on $n$. On the base interval, \eqref{eq:base-density} gives both statements. For
		$n>N_1$, put $q=q(n)$ and $m=n-q-1$. By the induction hypothesis, the tail contributes
		every integer in $[0,q-k\rho(q)]$, and
		\[
		q-k\rho(q)\ge\eta q\ge A L_{k-1}(n)+A\ge A L_{k-1}(m)+A.
		\]
		Thus its interval is long enough to smooth the main component's contribution set. Lemmas~\ref{lem:dense-contributions} and~\ref{lem:sumset} show that the new contribution sum contains every integer from $0$ to
		\[
		(m-k+1)+(q-k\rho(q))=n-k(1+\rho(q)).
		\]
		The new family has exactly $n$ vertices when selectors are included, and its component count
		is
		\begin{equation}
			\rho(n)=1+\rho(q(n)).\label{eq:component-recursion}
		\end{equation}
		Finally, the density estimate follows directly from the smaller instance and \eqref{eq:cutoff}:
		\begin{align*}
			n-k\rho(n)&=n-k-k\rho(q)\\
			&\ge n-k-(1-\eta)q\\
			&\ge\frac{1+\eta}{2}n-k\\
			&\ge\eta n.
		\end{align*}
		This proves admissibility and closes the induction without using any estimate for the recursion
		depth.
		
		It remains to count the components. Let $R$ be the largest number of components in any
		of our finitely many base tails. Along the recursion, write $x_0=n$ and $x_{j+1}=q(x_j)$ until
		a base argument is reached. By \eqref{eq:q-two-steps} and monotonicity, as long as the indicated recursive
		arguments exist,
		\[
		x_{2j}\le L^{\circ j}(n).
		\]
		By definition, $L^{\circ L_k(n)}(n)\le K\le N_1$. Therefore at most $2L_k(n)$ recursive steps are needed.
		Each step adds one component, so \eqref{eq:component-recursion} gives
		\[
		\rho(n)\le 2L_k(n)+R.
		\]
	\end{proof}
	
	The remaining lemmas will be used for the lower bound. They hold for every
	fixed $k\ge3$ and do not require assumption \eqref{eq:induction-assumption}.
	We use Gao's insertion tree~\cite{Gao}, which we now define.
	
	Let $H$ be a $k$-uniform hypergraph on $n\ge k$ vertices. Suppose that $H$ has $r=n-C$ distinct sizes of maximal cliques. Choose one maximal clique of each size and write them as
	\[
	|X_0|>|X_1|>\cdots>|X_{r-1}|.
	\]
	For $0\le i<r$, put $D_i=V(H)\setminus X_i$. There are $r-i$ distinct positive clique sizes at most $|X_i|$, so
	\begin{equation}
		|X_i|\ge r-i,\qquad |D_i|\le C+i. \label{eq:complement-bound}
	\end{equation}
	Every maximal clique has size at least $k-1$. Thus $r\le n-k+2$, and hence $C\ge k-2\ge1$.
	
	Associate a node $x_i$ with each $X_i$. Begin with the root $x_0$, and set
	\[
	A_{x_0}=X_0,\qquad B_{x_0}=D_0.
	\]
	Suppose that $x_0,\ldots,x_{i-1}$ have been inserted. To insert $x_i$, start at the root $u=x_0$. If $u$ has a child $w$ such that
	\[
	X_w\cap B_u=X_i\cap B_u,
	\]
	replace $u$ by $w$ and repeat. If there is no such child, attach $x_i$ as a new child of $u$ and define
	\begin{equation}
		A_{x_i}=A_u\cap X_i,\qquad B_{x_i}=A_u\setminus X_i. \label{eq:insertion-sets}
	\end{equation}
	Here and below, $X_{x_i}$ means $X_i$, and similarly for $D_{x_i}$. Children of a fixed node have different intersections with its $B$-set, so the child used in a descent is unique. Each descent moves strictly down the finite tree already constructed. Thus the procedure stops and defines a rooted tree on all $r$ nodes. Every parent has a smaller insertion index than its children.
	
	Consider a root path
	\[
	y_0,y_1,\ldots,y_d,\qquad y_0=x_0,
	\]
	and define
	\begin{equation}
		S_i=X_{y_i}\cap B_{y_{i-1}}\qquad (1\le i\le d). \label{eq:signatures}
	\end{equation}
	Whenever a later node is inserted below $y_i$, the insertion procedure passes from $y_{i-1}$ to $y_i$ by testing this intersection. It follows that
	\begin{equation}
		X_{y_j}\cap B_{y_{i-1}}=S_i\qquad (1\le i\le j\le d). \label{eq:signature-persistence}
	\end{equation}
	Repeatedly applying \eqref{eq:insertion-sets} gives the disjoint partition
	\[
	V(H)=A_{y_j}\sqcup B_{y_j}\sqcup B_{y_{j-1}}\sqcup\cdots\sqcup B_{y_0}.
	\]
	Since $X_{y_j}\cap B_{y_j}=\varnothing$, this partition and \eqref{eq:signature-persistence} imply
	\begin{equation}
		X_{y_j}=A_{y_j}\sqcup S_1\sqcup\cdots\sqcup S_j. \label{eq:path-partition}
	\end{equation}
	We call the sets $S_i$ the path signatures. They are pairwise disjoint. Moreover, every $S_j$ is nonempty: otherwise \eqref{eq:path-partition} and $A_{y_j}\subseteq A_{y_{j-1}}$ would imply $X_{y_j}\subseteq X_{y_{j-1}}$, contrary to the fact that these are distinct maximal cliques. Notice also that
	\begin{equation}
		S_1\subseteq D_0,\qquad S_i\subseteq B_{y_{i-1}}\subseteq A_{y_{i-2}}\subseteq X_0\qquad (i\ge2). \label{eq:signature-location}
	\end{equation}
	A tuple with one vertex from each of $S_1,\ldots,S_t$ is called a
	\emph{transversal $t$-tuple}. We next prove a missing-edge property
	and use it to bound the depth.
	
	\begin{lemma}[A missing edge along a path]
		\label{lem:missing-edge}
		Let $y_0,\ldots,y_{k-1}$ be a root path in the insertion tree. For every $x\in B_{y_{k-1}}$, there are $s_i\in S_i$ for $1\le i\le k-1$ such that
		\[
		\{x,s_1,\ldots,s_{k-1}\}\notin E(H).
		\]
	\end{lemma}
	
	\begin{proof}
		Since $x\notin X_{y_{k-1}}$ and $X_{y_{k-1}}$ is maximal, there is a set $Z\subseteq X_{y_{k-1}}$ of size $k-1$ for which $\{x\}\cup Z$ is not an edge.
		
		Fix $i\in\{1,\ldots,k-1\}$. We claim that
		\[
		x\in X_{y_{i-1}},\qquad X_{y_{k-1}}\setminus S_i\subseteq X_{y_{i-1}}.
		\]
		Indeed, $x\in B_{y_{k-1}}\subseteq A_{y_{i-1}}$, and $A_{y_{k-1}}\subseteq A_{y_{i-1}}$. For $j<i$, the set $S_j$ is contained in $X_{y_{i-1}}$ by persistence. For $j>i$, we have
		\[
		S_j\subseteq B_{y_{j-1}}\subseteq A_{y_{j-2}}\subseteq A_{y_{i-1}}.
		\]
		The claim therefore follows from \eqref{eq:path-partition}. If $Z$ missed $S_i$, the missing edge $\{x\}\cup Z$ would be contained in the clique $X_{y_{i-1}}$, a contradiction. Thus $Z$ meets all $k-1$ pairwise disjoint signatures. Since $|Z|=k-1$, it contains exactly one vertex from each of them.
	\end{proof}
	
	\begin{lemma}[Depth and degree bounds]
		\label{lem:depth-degree}
		Every node has depth at most $k-1$, and
		\[
		\deg_T(x_i)\le2^{C+i}\qquad (0\le i<r).
		\]
	\end{lemma}
	
	\begin{proof}
		Suppose that a root path $y_0,\ldots,y_k$ existed. Choose $x\in S_k$, which is possible because signatures are nonempty. Since $S_k\subseteq B_{y_{k-1}}$, Lemma~\ref{lem:missing-edge} gives a missing edge $\{x,s_1,\ldots,s_{k-1}\}$ with $s_i\in S_i$. All these vertices belong to $X_{y_k}$ by \eqref{eq:path-partition}, a contradiction.
		
		For each node $u$, its children have distinct nonempty intersections with $B_u$. Thus the root has at most $2^{|B_{x_0}|}-1$ children, and a nonroot node $u$ has degree at most $1+(2^{|B_u|}-1)=2^{|B_u|}$. Finally, $B_{x_i}\subseteq D_i$, so \eqref{eq:complement-bound} proves the bound.
	\end{proof}
	
	A node $v$ is a \emph{carrier} of a vertex set $Q$ if $Q\subseteq X_v$. When $Q$ has a carrier among the selected cliques, write $\operatorname{fc}(Q)$ for the carrier of smallest insertion index.
	
	\begin{lemma}[First carriers]
		\label{lem:first-carriers}
		Let $Q=\{p,q_2,\ldots,q_t\}$ consist of distinct vertices, where $p\in D_0$ and $q_2,\ldots,q_t\in X_0$. If $Q$ has a carrier, then
		\[
		\operatorname{depth}(\operatorname{fc}(Q))\le t.
		\]
		If equality holds, then on the root path to $\operatorname{fc}(Q)$ we have $p\in S_1$, and $\{q_2,\ldots,q_t\}$ meets each of $S_2,\ldots,S_t$ in exactly one vertex.
	\end{lemma}
	
	\begin{proof}
		Let $y_0,\ldots,y_d=v$ be the root path to $v=\operatorname{fc}(Q)$. Since $p\notin X_0$, we have $d\ge1$. By \eqref{eq:path-partition} and \eqref{eq:signature-location}, the membership $p\in X_v$ forces $p\in S_1$. Thus $p$ belongs to all of $X_{y_1},\ldots,X_{y_d}$.
		
		Consider a vertex $q\in Q\setminus\{p\}$. If $q\in A_{y_d}$, then $q$ belongs to every clique on the path. Otherwise the path decomposition places $q$ in exactly one $S_h$, with $h\ge2$. In this case
		\[
		q\in B_{y_{h-1}}\subseteq A_{y_{h-2}},
		\]
		so $q$ belongs to every clique up to $X_{y_{h-2}}$, is absent from $X_{y_{h-1}}$, and belongs to all subsequent cliques by persistence. Hence each of $q_2,\ldots,q_t$ is absent from at most one clique among $X_{y_1},\ldots,X_{y_{d-1}}$.
		
		Every ancestor of $v$ has an earlier insertion index. Therefore none of $y_1,\ldots,y_{d-1}$ carries $Q$. Each already contains $p$, so each must omit at least one of the remaining $t-1$ vertices. It follows that $d-1\le t-1$, proving $d\le t$. If $d=t$, each of the $t-1$ cliques $X_{y_1},\ldots,X_{y_{d-1}}$
		misses a different one of these vertices. The vertex absent at $y_{h-1}$ belongs to $S_h$. Thus the equality statement follows as well.
	\end{proof}
	
	Call nodes of depths $1,\ldots,k-2$ \emph{structural nodes}. Nodes of depth $k-1$ are leaves by Lemma~\ref{lem:depth-degree}. For $1\le t\le k-2$, a depth-$t$ node $u$ is \emph{$t$-critical} if there is a choice
	\[
	(s_1,\ldots,s_t)\in S_1(u)\times\cdots\times S_t(u)
	\quad\text{such that}\quad u=\operatorname{fc}(\{s_1,\ldots,s_t\}).
	\]
	Define the set of \emph{level-$t$ resets} by
	\begin{equation}
		R_t=\{u:1\le\operatorname{depth}(u)<t\}\cup\{u:u\text{ is $t$-critical}\}. \label{eq:resets}
	\end{equation}
	Then $R_1\subseteq R_2\subseteq\cdots\subseteq R_{k-2}$. Every $1$-critical node is the first carrier of some vertex in $D_0$. Different such nodes must use different vertices, and therefore
	\begin{equation}
		|R_1|\le|D_0|\le C. \label{eq:first-resets}
	\end{equation}
	We will repeatedly use the following consequence of Lemma~\ref{lem:first-carriers}: the first carrier of any transversal $t$-tuple lies in $R_t$. Here the tuple may be selected on a different root path from the path to its first carrier. Its first vertex lies in $D_0$ and all its other vertices lie in $X_0$, so Lemma~\ref{lem:first-carriers} applies. If its first carrier has depth less than $t$, it lies in the first set in \eqref{eq:resets}; if its depth is $t$, the equality statement makes it $t$-critical on its own root path.
	
	List the structural nodes as $w_1,\ldots,w_s$ in increasing insertion order. For a node $u$ of depth $k-2$, let $\ell(u)$ be its number of children; all these children have depth $k-1$. After processing the first $j$ structural nodes, define
	\begin{equation}
		U_j=1+j+\sum_{\substack{u\in\{w_1,\ldots,w_j\}\\ \operatorname{depth}(u)=k-2}}\ell(u),\qquad 0\le j\le s. \label{eq:virtual-count}
	\end{equation}
	Thus a deepest structural node and all its leaf children are counted when that structural node is processed, even if some of those children have much later insertion indices.
	
	\begin{lemma}[Virtual counts]
		\label{lem:virtual-count}
		We have $U_0=1$, $U_s=r$, and $U_j\le r$ for every $j$. If $w_{j+1}=x_i$, then $i\le U_j$.
	\end{lemma}
	
	\begin{proof}
		Every node is either the root, a structural node, or a depth-$(k-1)$ leaf. A leaf has exactly one deepest structural parent. Thus all terms in \eqref{eq:virtual-count} count disjoint sets of nodes, and the final count includes every node exactly once.
		
		Consider a node whose insertion index is smaller than $i$. If it is structural, it has already been processed. If it is a depth-$(k-1)$ leaf, its parent has an even smaller insertion index and has already been processed, so this leaf is already counted virtually. The root is also counted. Therefore all $i$ nodes preceding $x_i$ are among the $U_j$ virtually counted nodes.
	\end{proof}
	
	A \emph{level-$t$ quiet segment} begins at any processing boundary, and ends immediately after the next processed node in $R_t$, including that node. If there is no such subsequent node, it ends at the final boundary. In particular, the open interior of a quiet segment contains no level-$t$ reset. Only nonempty segments will need to be considered. When there is no endpoint reset, all remaining processed nodes belong to the part before the reset of the segment.
	
	\begin{lemma}[Counting leaf children]
		\label{lem:leaf-count}
		Let $u$ be a depth-$(k-2)$ node. For each leaf child $v$ of $u$, put $P_v=X_v\cap B_u$. For every transversal prefix $\mathbf{s}=(s_1,\ldots,s_{k-2})$ on the path to $u$, choose a carrier $\lambda(\mathbf{s})$. Suppose that a set $E\subseteq V(H)$ contains every $D_{\lambda(\mathbf{s})}$, and write $e=|E|$. If two leaf children $v,v'$ satisfy $P_v\cap E=P_{v'}\cap E$, then
		\[
		B_v\setminus E=B_{v'}\setminus E.
		\]
		Thus,
		\begin{equation}
			\ell(u)\le2^e\bigl(|B_u|+e+1\bigr). \label{eq:leaf-count}
		\end{equation}
	\end{lemma}
	
	\begin{proof}
		Fix a child $v$, and write $\sigma=P_v\cap E$. If $x\in B_v\setminus E$, Lemma~\ref{lem:missing-edge} applied to the root path ending at $v$ gives a transversal prefix $\mathbf{s}=(s_1,\ldots,s_{k-2})$ and $z\in P_v$ such that
		\begin{equation}
			\{x,s_1,\ldots,s_{k-2},z\}\notin E(H). \label{eq:missing-edge}
		\end{equation}
		The chosen carrier $\lambda(\mathbf{s})$ contains every $s_i$. It also contains $x$, because $x\notin E\supseteq D_{\lambda(\mathbf{s})}$. Therefore it cannot contain $z$, and hence $z\in D_{\lambda(\mathbf{s})}\subseteq E$. Thus $z\in\sigma$.
		
		Conversely, suppose $x\in A_u\setminus E$ and that \eqref{eq:missing-edge} holds for some transversal prefix $\mathbf{s}$ and some $z\in\sigma$. The vertices $s_1,\ldots,s_{k-2},z$ all belong to $X_v$. Since $X_v$ is a clique, it follows that $x\notin X_v$, and therefore $x\in A_u\setminus X_v=B_v$. We have proved the exact description
		\[
		B_v\setminus E=\left\{x\in A_u\setminus E:
		\begin{array}{l}
			\text{there exist a transversal prefix $\mathbf{s}$ and $z\in\sigma$}\\
			\text{for which \eqref{eq:missing-edge} holds}
		\end{array}\right\}.
		\]
		The right-hand side depends only on $u$ and $\sigma$, proving the first assertion.
		
		It remains to count possible clique sizes. The path identities give
		\[
		D_v=(D_u\setminus P_v)\sqcup B_v.
		\]
		Since $P_v\subseteq B_u\subseteq D_u$, put $c_u=|D_u\setminus B_u|$ to obtain
		\begin{equation}
			|D_v|=c_u+|B_u\setminus P_v|+|B_v\setminus E|+|B_v\cap E|. \label{eq:complement-count}
		\end{equation}
		For a fixed pattern $\sigma$, the term $|B_v\setminus E|$ is fixed. The sum of the other two variable terms lies between $0$ and $|B_u|+e$, so $|D_v|$ has at most $|B_u|+e+1$ possible values. There are at most $2^e$ patterns. Since the selected cliques have pairwise distinct sizes, their complement sizes are also pairwise distinct. This proves \eqref{eq:leaf-count}.
	\end{proof}
	
	All constants in the remainder of the section may depend on $k$. For a processing boundary
	with virtual count $T$, define
	\begin{equation}
		E_T=\bigcup_{0\le i\le\min\{T,r-1\}}D_i.
		\label{eq:early-complements}
	\end{equation}
	If $T\ge C$ and $T\ge1$, then \eqref{eq:complement-bound} yields
	\begin{equation}
		|E_T|\le(T+1)(C+T)\le4T^2.
		\label{eq:early-complement-bound}
	\end{equation}
	By Lemma~\ref{lem:virtual-count}, every structural node processed by this boundary has insertion index at most
	$T$.
	
	\begin{lemma}[Deepest quiet segments]
		\label{lem:deepest-segment}
		For every fixed $k\ge3$, there are constants $K_k$ and $T_k$ such that the following holds.
		If a level-$(k-2)$ quiet segment starts with virtual count $T\ge\max\{C,T_k\}$ and ends with
		count $U'$, then
		\[
		L_3(U')\le L_3(T)+K_k.
		\]
	\end{lemma}
	
	\begin{proof}
		Every structural node in the open interior of the segment has depth $k-2$ and is
		noncritical. Consider such a node $u$ and a transversal prefix of length $k-2$ on its root
		path. Its first carrier belongs to $R_{k-2}$ by Lemma~\ref{lem:first-carriers}. Since $u$ itself carries the prefix and is
		noncritical, that first carrier has a strictly earlier insertion index than $u$. It cannot occur in
		the open interior of the quiet segment. Thus it was processed before the segment began, and
		its complement is contained in $E_T$. Carrier compression therefore applies at every such $u$
		with the same set $E=E_T$.
		
		We first bound the number $N$ of these interior nodes. If $k\ge4$, their parents have depth
		$k-3$ and belong to $R_{k-2}$. Each such parent was processed before the segment began. There
		are at most $T$ such parents, each with insertion index at most $T$, and hence each has degree
		at most $2^{C+T}$. If $k=3$, every structural node is a child of the root, so directly $N\le2^C$. In
		both cases, after increasing a constant if necessary,
		\begin{equation}
			N\le T2^{C+T}\le2^{a_kT}.
			\label{eq:deepest-node-count}
		\end{equation}
		Suppose that an interior node $u$ is about to be processed, and the current virtual count is
		$U$. Its actual insertion index is at most $U$, so $|B_u|\le C+U$. By \eqref{eq:early-complement-bound} and \eqref{eq:leaf-count},
		\[
		\ell(u)\le2^{4T^2}(C+U+4T^2+1).
		\]
		Since $U\ge T\ge C$, there is a constant $b_k$ such that the new virtual count satisfies
		\begin{equation}
			U_{\mathrm{new}}+1\le2^{b_kT^2}(U+1).
			\label{eq:virtual-update}
		\end{equation}
		Iterating this inequality at most $2^{a_kT}$ times gives, just before an endpoint reset (or at the
		final boundary if there is no reset),
		\[
		U_{\mathrm{pre}}+1\le(T+1)2^{b_kT^2 2^{a_kT}}\le2^{2^{c_kT}}
		\]
		for sufficiently large $T$.
		
		If the endpoint is a shallower structural node, processing it increases the virtual count
		by only $1$. If it is a deepest critical node, its insertion index is at most $U_{\mathrm{pre}}$, and the degree bound gives
		\[
		U'\le U_{\mathrm{pre}}+1+2^{C+U_{\mathrm{pre}}}.
		\]
		Thus in every case $U'$ is bounded by an exponential tower of fixed height with bottom
		parameter a constant multiple of $T$. Rank stability, Lemma~\ref{lem:ranks}, now implies
		$L_3(U')\le L_3(T)+O_k(1)$.
	\end{proof}
	
	\begin{lemma}[Counting resets at the next level]
		\label{lem:next-resets}
		Suppose $1\le t\le k-3$. A level-$t$ quiet segment starting with virtual count
		$T\ge\max\{C,T_k\}$ contains at most $2^{a_kT}$ nodes of $R_{t+1}$, for a suitable constant $a_k$.
	\end{lemma}
	
	\begin{proof}
		In the open interior of the segment there are no nodes of depth less than $t$, and
		no $t$-critical nodes. Therefore a node of $R_{t+1}$ in this interior is either a depth-$t$ node or a
		$(t+1)$-critical depth-$(t+1)$ node. The endpoint can contribute at most one additional reset.
		
		Let $N_t$ be the number of depth-$t$ nodes in the open interior. For $t\ge2$, each has a
		depth-$(t-1)$ parent in $R_t$. That parent must have been processed before the segment began.
		There are at most $T$ possible parents, and their indices are at most $T$. For $t=1$, the parent
		is always the root. Thus in either case
		\begin{equation}
			N_t\le T2^{C+T}\le2^{b_kT}.
			\label{eq:depth-t-count}
		\end{equation}
		Now consider an interior $(t+1)$-critical node $v$ with parent $u$ of depth $t$. Choose a
		witnessing tuple
		\[
		(s_1,\ldots,s_t,z),\qquad v=\operatorname{fc}(\{s_1,\ldots,s_t,z\}).
		\]
		The parent $u$ contains the prefix $\{s_1,\ldots,s_t\}$. Its first carrier $\lambda$ therefore precedes $v$ and
		belongs to $R_t$. Since the segment has no earlier interior $R_t$ node, $\lambda$ was processed before
		the segment began. Hence $D_\lambda\subseteq E_T$. If $z$ belonged to $X_\lambda$, then $\lambda$ would carry the full tuple
		before $v$, a contradiction. Therefore $z\in E_T$.
		
		For a fixed parent $u$, distinct critical children must receive distinct vertices $z$ in these
		choices. Indeed, if two such children received the same $z$, the earlier child would contain that
		vertex and all the parent signatures $S_1(u),\ldots,S_t(u)$. It would therefore carry the witnessing
		tuple of the later child, contradicting the definition of its first carrier. Thus a fixed parent
		has at most $|E_T|\le4T^2$ critical children in the segment.
		
		A relevant depth-$t$ parent was either processed before the segment began or is one of its
		$N_t$ interior depth-$t$ nodes. There are at most $T+N_t$ such parents. Including the possible
		endpoint reset, the total number of $R_{t+1}$ nodes is at most
		\[
		N_t+(T+N_t)4T^2+1\le2^{a_kT},
		\]
		after adjusting $a_k$ and the fixed threshold.
	\end{proof}
	
	\begin{lemma}[Rank bound for quiet segments]
		\label{lem:quiet-rank}
		For every fixed $k\ge3$, there are constants $K_k,T_k$ such that, for each
		$1\le t\le k-2$, a level-$t$ quiet segment beginning with virtual count
		$T\ge\max\{C,T_k\}$ and ending with $U'$ satisfies
		\[
		L_{k-t+1}(U')\le L_{k-t+1}(T)+K_k.
		\]
	\end{lemma}
	
	\begin{proof}
		We use downward induction on $t$. The case $t=k-2$ is Lemma~\ref{lem:deepest-segment}. When $k=3$, this
		is the only case.
		
		Let $1\le t\le k-3$, and assume the claim at level $t+1$. Cut a level-$t$ quiet segment
		immediately after each of its $R_{t+1}$ nodes. Each nonempty piece is a level-$(t+1)$ quiet segment:
		the level-$t$ endpoint, if it is a reset, also lies in $R_{t+1}$ because $R_t\subseteq R_{t+1}$; a final piece without
		a reset reaches the final processing boundary. By Lemma~\ref{lem:next-resets}, the number $N$ of pieces is at
		most $2^{a_kT}+1$, and hence at most $2^{a'_kT}$ after changing the constant.
		
		Write their consecutive endpoint counts as
		\[
		T=V_0\le V_1\le\cdots\le V_N=U'.
		\]
		All starting counts are at least $T$, so the induction hypothesis applies to every piece. Put
		$m=k-t\ge3$. It gives
		\[
		L_m(V_{i+1})\le L_m(V_i)+K_k\qquad(0\le i<N).
		\]
		Rank aggregation, Lemma~\ref{lem:combine-ranks}, yields
		\[
		L_{m+1}(U')\le L_{m+1}(T)+O_k(1).
		\]
		Since $m+1=k-t+1$, this is the required statement. There are only $k-2$ levels, so the
		constants and thresholds may be enlarged to work at every level at the same time.
	\end{proof}
	
	\section{Proof of the main theorem}
	\label{sec:proof}
	
	\begin{proof}[Proof of Theorem~\ref{thm:main}]
		We first prove \(f(n,k)=O_k(L_k(n))\) by induction on \(k\).
		Lemma~\ref{lem:graphs} gives \(f(n,2)=O(L_2(n))\), which starts the induction.
		For \(k\ge3\), assume the upper bound for \(k-1\). We can then choose
		\(A\) so that \eqref{eq:induction-assumption} holds.
		Lemma~\ref{lem:tails} gives an admissible \(n\)-tail with
		\(\rho(n)\le2L_k(n)+R\) components for every sufficiently large \(n\).
		By Lemma~\ref{lem:selectors},
		\[
		f(n,k)\le k\rho(n)-1\le2kL_k(n)+kR-1.
		\]
		Since \(L_k(n)\) tends to infinity, this is the required upper bound.
		It also proves the induction step when \(k=3\).
		
		We next prove the lower bound. Let $H$ be any $n$-vertex $k$-uniform
		hypergraph with $r=n-C$ distinct maximal-clique sizes. Choose the cliques
		and form the insertion tree as in Section~\ref{sec:lemmas}.
		We will show that $L_k(n)=O_k(C)$. Increase the threshold in Lemma~\ref{lem:quiet-rank} so that $T_k\ge2$, and put
		\[
		T_0=\max\{C,T_k\}.
		\]
		If the final virtual count is less than $T_0$, then $r<T_0$ and $L_k(r)=O_k(C)$, since $C\ge1$.
		
		Otherwise let $U^*$ be the first virtual count at least $T_0$, and let $U^-<T_0$ be its predecessor.
		The node processed at this step has insertion index at most $U^-$ by Lemma~\ref{lem:virtual-count}. If it is
		shallower than depth $k-2$, the count increases by $1$. If it has depth $k-2$, its number of
		children is at most $2^{C+U^-}$. Therefore, in either case,
		\[
		U^*\le T_0+1+2^{C+T_0}\le2^{O_k(T_0)}.
		\]
		By rank stability and $T_0=O_k(C)$,
		\begin{equation}
			L_k(U^*)\le L_k(T_0)+O_k(1)=O_k(C).
			\label{eq:initial-rank}
		\end{equation}
		Starting at this boundary, partition the rest of the structural process into level-$1$ quiet
		segments. By \eqref{eq:first-resets}, there are at most $C+1$ nonempty segments. Each begins with a count
		at least $T_0$, so Lemma~\ref{lem:quiet-rank} with $t=1$ shows that each increases $L_k$ by at most a constant
		depending on $k$. Since the final virtual count is $r$, we conclude that
		\[
		L_k(r)\le L_k(U^*)+O_k(C+1)=O_k(C).
		\]
		This also holds in the case $r<T_0$ already considered.
		
		If $C>n/2$, then $L_k(n)=O_k(C)$ follows immediately from $L_k(n)=o(n)$. Otherwise
		$n\le2r$, and monotonicity and rank stability give
		\[
		L_k(n)\le L_k(2r)\le L_k(r)+O_k(1)=O_k(C).
		\]
		Choose $H$ with $g(n,k)$ distinct maximal-clique sizes. Then
		$C=n-g(n,k)=f(n,k)$, so $f(n,k)=\Omega_k(L_k(n))$.
		Together with the upper bound, this gives $f(n,k)=\Theta_k(L_k(n))$.
		
		It remains to estimate $c(n,k)$. Write \(C=c(n,k)\). If \(C>n/2\), then
		\(C\ge\log_2 L_k(n)-O_k(1)\) immediately. Otherwise
		\[
		n\le2(n-C)\le2N_0(k-1,C).
		\]
		Lemma~\ref{lem:tree-size} and rank stability give \(L_k(n)=O_k(2^C)\), so
		\(C\ge\log_2 L_k(n)-O_k(1)\).
		The bound in Lemma~\ref{lem:tree-size} may be extended to the finitely many small
		values of \(C\) by increasing its constant.
		
		Conversely, put
		\[
		C'=\lceil\log_2 L_k(n)\rceil+A_k,
		\]
		where the fixed integer \(A_k\) is sufficiently large. The lower estimate
		in Lemma~\ref{lem:tree-size} gives
		\(L_k(N_0(k-1,C'))>L_k(n)\).
		Since \(L_k\) is nondecreasing, this implies \(N_0(k-1,C')>n\). Thus
		\(n-C'\le N_0(k-1,C')\) and \(c(n,k)\le C'\).
		
		We have proved \(c(n,k)=\log_2 L_k(n)+O_k(1)\). Taking powers of \(2\)
		gives
		\[
		2^{c(n,k)}=\Theta_k(L_k(n)).
		\]
		Since \(f(n,k)=\Theta_k(L_k(n))\), we also have
		\(f(n,k)=\Theta_k(2^{c(n,k)})\).
	\end{proof}

\end{document}